\documentclass[12pt,a4paper]{article}
\usepackage{amssymb,amsmath,url}

\def\pg{\mbox{\rm PG}}
\def\ag{\mbox{\rm AG}}

\def\q{\mbox{\rm Q}}

\def\mod{\mbox{\rm{ mod }}}

\def\C{\mathcal{C}}
\def\S{\mathcal{S}}
\def\P{\mathcal{P}}
\def\B{\mathcal{B}}
\def\I{\mathrel{\mathrm I}}
\def\O{\mathcal{O}}

\def\K{\mathcal{K}}

\def\IS{\S=(\P,\B,\I)}
\newenvironment{proof}{\noindent{\bf Proof. }}{\ignorespaces\rule{1pt}{0pt}\hfill $\square$\medskip\smallskip\smallskip}
\newtheorem{lemma}{Lemma}
\newtheorem{theorem}{Theorem}
\newtheorem{corollary}{Corollary}

\title{On large maximal partial ovoids of the parabolic quadric $\q(4,q)$\footnote{final version, appeared in Des. Codes Cryptogr., DOI:10.1007/s10623-012-9629-y}}
\author{Jan De Beule\thanks{The author is a postdoctoral research fellow of the Research Foundation Flanders -- Belgium (FWO).} }
\date{}

\begin{document}
\maketitle

\begin{abstract}
We use the representation $T_2(\O)$ for $\q(4,q)$ to show that maximal partial ovoids of 
$\q(4,q)$ of size $q^2-1$, $q=p^h$, $p$ odd prime, $h > 1$, do not exist. Although this was known 
before, we give a slightly alternative proof, also resulting in more combinatorial information of the 
known examples for $q$ odd prime.
\end{abstract}

{\bf keywords}: maximal partial ovoid, generalized quadrangle, parabolic quadric.


\medskip
{\em This paper is based on joint work with Andr\'as G\'acs, started in the autumn of 2008, as a 
continuation of the work in \cite{DeBeuleGacs}.  His unfortunate and sudden death prevented us
from continuing our joint work on the geometrical interpretation of the results. I would like to dedicate this work to Andr\'as.}

\section{Introduction}

A (finite) \emph{generalized quadrangle} (GQ) is an incidence structure
$\IS$ in which $\P$ and $\B$ are disjoint non-empty sets of objects called
points and lines, and for which $\I \subseteq (\P \times
\B) \cup (\B \times \P)$ is a symmetric point-line incidence relation
satisfying the following axioms:
\begin{itemize}
\item[(i)] each point is incident with $1+t$ lines $(t \geqslant 1)$ and
two distinct points are incident with at most one line;
\item[(ii)] each line is incident with $1+s$ points $(s \geqslant 1)$ and
two distinct lines are incident with at most one point;
\item[(iii)] if $x$ is a point and $L$ is a line not incident with $x$,
then there is a unique pair $(y,M) \in \P \times \B$ for which $x \I M \I
y \I L$.
\end{itemize}
The integers $s$ and $t$ are the parameters of the GQ and $\S$ is said to
have order $(s,t)$. If $s=t$, then $\S$ is said to have order $s$. If $\S$
has order $(s,t)$, then $|\P| = (s+1)(st+1)$ and $|\B| =
(t+1)(st+1)$ (see e.g. \cite{PayneThas84}). 

An {\em ovoid} of a GQ $\S$ is a set $\O$ of points of $\S$ such that every line
is incident with exactly one point of the ovoid. An ovoid of a GQ of order
$(s,t)$ has necessarily size $1+st$. A {\em partial ovoid} of a
GQ is a set $\K$ of points such that every line contains {\em at most} one point
of $\K$. A partial ovoid $\K$ is called {\em maximal} if and only if $\K \cup
\{P\}$ is not a partial ovoid for any point $P \in \P \setminus \K$, in other words, if
$\K$ cannot be extended. It is clear that any partial ovoid of a GQ of order
$(s,t)$ contains $1+st-\rho$ points, $\rho \geq 0$, with $\rho = 0$ if and only
if $\K$ is an ovoid. 

It is a natural question to study {\em extendability} of partial ovoids, i.e. can one alway extend a partial ovoid of size $1+st-\epsilon$ (e.g. to an ovoid) if $\epsilon$ is not too big? The following theorem is a typical example. 

\begin{theorem}[{\cite[2.7.1]{PayneThas84}}]\label{theo1}
Let $\IS$ be a GQ of order $(s,t)$. Any partial ovoid of size $st-\rho$, $0 \leq \rho < \frac{t}{s}$ is contained in a uniquely defined ovoid of $\S$.
\end{theorem}

Note that if no ovoids of a particular GQ exist, then Theorem~\ref{theo1} implies an upper bound on the size of partial ovoids. 
The following theorem deals with the limit situation, and will be of use in Section~\ref{sec:geom}.

\begin{theorem}[{\cite[2.7.2]{PayneThas84}}]\label{theo2}
Let $\IS$ be a GQ of order $(s,t)$. Let $\K$ be a maximal partial ovoid of size
$st-t/s$ of $\S$. Let $\B'$ be the set of lines incident with no point
of $\K$, and let $\P'$ be the set of points on at least one line of $\B'$ and
let $\I'$ be the restriction of $\I$ to points of $\P'$ and lines of $\B'$. Then
$\S'=(\P',\B',\I')$ is a subquadrangle of order $(s,\rho=t/s)$.
\end{theorem}

Consider the parabolic quadric $\q(4,q)$ in the $4$-dimensional projective space $\pg(4,q)$. 
This quadric is the set of points and lines that are totally isotropic with relation to a non-singular 
quadratic form on $\pg(4,q)$, which is, up to coordinate transform, unique, and its points and 
lines constitute an example of a generalized quadrangle of order $q$. 

It is known, (see e.g. \cite{PayneThas84}) that this GQ has
ovoids. A particular example of an ovoid is any elliptic quadric $\q^-(3,q)$
contained in it and obtained by a hyperplane section of $\q(4,q)$.  
When $q$ is prime, these are the only ovoids \cite{BallGovaertsStorme}; when $q$ 
is a prime power, other examples are known, see e.g. \cite{DeBeuleKleinMetsch11} 
for a list of references. The classification of ovoids of $\q(4,q)$, for $q$ prime, is essentially due to
the computation of intersection numbers (modulo $p$) of a hypothetical ovoid with elliptic quadrics, and the use of this information in a combinatorial argument.

Applying Theorem~\ref{theo1} to the $GQ$ $\q(4,q)$ implies that a partial ovoid
of size $q^2$ cannot be maximal. It is shown in \cite{DeBeuleGacs} that maximal partial
ovoids of $\q(4,q)$, $q=p^h$, $p$ an odd prime, $h > 1$, do not exist. The natural question arises
if maximal partial ovoids exist when $h=1$. Curiously, this is the case when $p \in \{3,5,7,11\}$, but 
no examples are known for $q > 11$, \cite{Penttila}. In this paper we give a slightly alternative 
proof  of the non-existence result for $h > 1$. Further, we compute the intersection numbers 
of a hypothetical maximal partial ovoid of size $q^2-1$ with elliptic quadrics embedded in $\q(4,q)
$, for $q$ an odd prime.  This yields structural information on the existing examples, and it is our 
hope that this information could contribute to finally proving their uniqueness and non-existence for
$p > 11$. 

\section{Non-existence for $q > p$}

We follow almost the same approach as in \cite{DeBeuleGacs}. Therefore we need to introduce an 
alternative representation of the GQ $\q(4,q)$. 

An {\em oval} of $\pg(2,q)$ is a set of $q+1$ points $\C$, such that no three points
of $\C$ are collinear. When $q$ is odd, it is known that all ovals of 
$\pg(2,q)$
are conics. When $q$ is even, several other examples and infinite families are
known, see e.g. \cite{Cheroweb}. The GQ $T_2(\C)$ is defined as follows. Let
$\C$ be an oval of $\pg(2,q)$, embed $\pg(2,q)$ as a plane in
$\pg(3,q)$ and denote this plane by $\pi_{\infty}$. Points are defined as
follows:
\begin{itemize}
\item[(i)] the points of $\pg(3,q) \setminus \pg(2,q)$;
\item[(ii)] the planes $\pi$ of $\pg(3,q)$ for which $|\pi \cap \C| =
1$; 
\item[(iii)] one new symbol $(\infty)$.
\end{itemize}
Lines are defined as follows:
\begin{itemize}
\item[(a)] the lines of $\pg(3,q)$ which are not contained in $\pg(2,q)$
and meet $\C$ (necessarily in a unique point);
\item[(b)] the points of $\C$.
\end{itemize}
Incidence between points of type (i) and (ii) and lines of type (a) and
(b) is the inherited incidence of $\pg(3,q)$. In addition, the point
$(\infty)$ is incident with no line of type (a) and with all lines of type
(b). It is straightforward to show that this incidence structure is a
GQ of order $q$. The following theorem (see e.g. \cite{PayneThas84}) allows us to
use this representation. 

\begin{theorem}\label{theo3}
The GQs $T_2(\C)$ and $\q(4,q)$ are isomorphic if and only if $\C$ is a conic of
the plane $\pg(2,q)$.
\end{theorem}

From now on we suppose that $\C$ is a conic. Let $\K$ be a maximal partial ovoid of size 
$k$ of $T_2(\C)$. Since $\q(4,q )\cong T_2(\C)$ has a collineation group acting transitively on the
points (see e.g. \cite{HirschfeldThas91}), we can suppose that $(\infty) \in \K$. This
implies that $\K$ contains no points of type (ii). It is clear that no two points of type (i) 
of $\K$ determine a line meeting $\pi_{\infty}$ in a point of $\C$. Hence the existence of 
$\K$ implies the existence of a set $U$ of $k-1$ points of type (i) such that no two 
points determine a line meeting $\pi_{\infty}$ in $\C$. It is easy to see that the 
converse is also true: from a set $U$ of $k-1$ points in $\pg(3,q) \setminus \pi_{\infty}$ with the 
property that all lines joining at least two points of $U$  are disjoint from $\C$, we can find a
partial ovoid $\K$ of $T_2(\C)$ of size $k$ by adding $(\infty)$ to $U$. 
The maximality of $\K$ is equivalent to the maximality of $U$. 

Hence the existence of a maximal partial ovoid of size $q^2-1$ of $\q(4,q)$, is equivalent with 
the existence of a set $U$ of $q^2-2$ affine points, not determining the points of a conic 
at infinity. In \cite{DeBeuleGacs}, it is shown that such a set $U$ can always be extended 
when $q > p$.  In fact, only the assumption that at least $p+2$ points are not determined
is used. In this paper we will assume that the points of a conic are not determined and that $U$ is not extendable, 
compute the range of a certain polynomial, and find a contradiction when $q > p$.
In the third section, we will describe the use of this particular polynomial to compute the 
intersection numbers modulo $p$ of the point set $U$ with planes of $\ag(3,q)$. This will yield intersection 
numbers modulo $p$ of the maximal partial ovoid of size $q^2-1$ with elliptic 
quadrics embedded in $\q(4,q)$. 

From now on, let $\K$ denote a partial ovoid of $\q(4,q)$, $q=p^h$, $p$ an odd prime and $h \geq 1$. 
Let $U$ denote the point set of $\pg(3,q) \setminus \pi_{\infty}$ corresponding with the partial ovoid of $\q(4,q)$. 

\begin{lemma}\label{le:pigeon}
If a plane $\pi \neq \pi_{\infty}$ of $\pg(3,q)$, meets $\C$ in at least one point, then $|\pi \cap U | \leq q$.
\end{lemma}
\begin{proof}
Let $P \in \C \cap \pi$. If $|\pi \cap U | > q$, then at least one of the $q$ lines of 
$\pi$ through $P$ must contain two points of $U$. This contradicts the fact 
that $U$ does not determine any point of $\C$. Hence, $|\pi \cap U| \leq q$.
\end{proof}

We choose $\pi_{\infty}$ to be the plane with equation $X_3 = 0$. Then any line $l$ of $\pi_{\infty}$ 
is determined by the equation $yX_0 + zX_1 + wX_2 = 0$, $(y,z,w) \in \mathbb{F}_q^3 \setminus \{(0,0,0)\}$. 
We denote such a line as $l(y,z,w)$. Any plane $\pi \neq \pi_{\infty}$ through $l(y,z,w)$ is determined 
by the equation $yX_0 + zX_1 + wX_2 +xX_3= 0$. We denote such a plane as $\pi(x,y,z,w)$.

The point set $U=\{ (a_i,b_i,c_i,1):i=1,\dots ,q^2-2 \}\subset \pg(3,q) \setminus \pi_{\infty}$. We define the 
R\'edei polynomial associated to the point set $U$ as follows:
\[
R(X,Y,Z,W)=\prod _{i=1}^{q^2-2}(X+a_iY+b_iZ+c_iW)=
X^{q^2-2}+\sum _{i=1}^{q^2-2}\sigma _i(Y,Z,W)X^{q^2-2-i}\,.
\]
Here $\sigma _i(Y,Z,W)$ is the $i$-th elementary symmetric polynomial of the 
multi-set $\{ a_iY+b_iZ+c_iZ:i\}$ and is either zero or has degree $i$. 

\begin{lemma}\label{le:div}
Suppose that the line $l(y,z,w)$ meets $\C$ in at least one point. Then $R(X,y,z,w) \mid (X^q-X)^q$.
\end{lemma}
\begin{proof}
If $x \in \mathbb{F}_q$ is a root of $R(X,y,z,w) = 0$, then its multiplicity equals $|\pi(x,y,z,w) 
\cap U| \leq q$ (the latter by Lemma~\ref{le:pigeon}). Since $|U| = q^2-2$, each of the $q$ 
planes $\pi(x,y,z,w)$, $x \in \mathbb{F}_q$, contains points of $U$, hence $R(X,y,z,w)=0$ has each 
element $x \in \mathbb{F}_q$ as root, with multiplicity at most $q$, and the lemma follows.
\end{proof}

Since we suppose that $q$ is odd and $|U|=q^2-2$, after the affine translation 
\[ a_i \mapsto a_i - \frac{\sum a_i}{q^2-2}, \qquad
b_i \mapsto b_i - \frac{\sum b_i}{q^2-2}, \qquad
c_i \mapsto c_i - \frac{\sum c_i}{q^2-2}, \]
not affecting the (non)-determined points at infinity, we may assume that $\sum a_i = \sum b_i = \sum c_i = 0$, which is equivalent to $\sigma_1(Y,Z,W) \equiv 0$.  

\begin{lemma}\label{le:sigma2}
If a line $l(y,z,w)$ has at least one common point with $\C$, then 
\begin{equation}\label{eq:sigma2}
R(X,y,z,w)(X^2-\sigma _2(y,z,w))=(X^q-X)^q.
\end{equation}
\end{lemma} 
\begin{proof}
From Lemma~\ref{le:div} we know that
\[
R(X,y,z,w)(X-S)(X-S')=(X^q-X)^q \label{eq:lacu},
\]
where $S$ and $S'$ are not necessarily different and depend on $y,z,w$. 
Considering the first three terms on both sides and taking into account that
$\sigma_1(Y,Z,W) \equiv 0$, we have  $(X-S)(X-S')=X^2-\sigma _2(y,z,w)$. 
\end{proof}

\begin{lemma}\label{le:sigmas}
Suppose that the line $l(y,z,w)$ meets $\C$ in at least one point. Then
\[\sigma_{2l+1}(y,z,w) = 0, 0 \leq l \leq \frac{q^2-3}{2},\]
\[\sigma_{2l}(y,z,w) = \sigma_2^l(y,z,w), 0 \leq l \leq \frac{q^2-q-2}{2},\]
\[\sigma_{q^2-q+2k}(y,z,w) =
\sigma_2^{\frac{q^2-q+2k}{2}}(y,z,w)-\sigma_2^k(y,z,w), 0 \leq k \leq
\frac{q-3}{2}. \]
\end{lemma}
\begin{proof}
Computation of both right-hand and left-hand sides of Equation~\eqref{eq:sigma2}, 
and the use of $\sigma_1(Y,Z,W) \equiv 0$ proves the lemma.
\end{proof}

\begin{corollary}\label{cor:sigmas}
\[\sigma_{2l+1}(Y,Z,W) \equiv 0,\, 0 \leq l \leq \frac{q-1}{2},\]
\[\sigma_{2l}(Y,Z,W) \equiv \sigma_2^l(Y,Z,W), \,0 \leq l \leq \frac{q-1}{2}.\]
\end{corollary}
\begin{proof}
Consider any line $l(y,z,w)$ meeting $\C$ in at least one point. By Lemma~\ref{le:sigmas}, 
the equations of the corollary are true after subsituting $Y=y$, $Z=z$, $W=w$. 
But for each point $P \in \C$, each line $l(y,z,w)$ on $P$ gives a substitution for 
which the equations are true. Dually, this means that the points of at least $q+1$ 
different lines are a solution of the equations of the corollary. Since the degree of each equation 
is at most $q$, by the theorem of B\'ezout, each curve represented by an equation must 
contain $q+1$ lines as a component. But then its degree must be at least $q+1$. 
Hence, the polynomials are identically zero.
\end{proof}

We define now the polynomials $S_j$, $j=0,\ldots, q-1$ as follows.
\[
S_j(Y,Z,W) := \sum_{i=1}^{q^2-2} (a_iY+b_iZ+c_iW)^j\,.
\]
The Newton identities describe a relation between the polynomials $S_j(Y,Z,W)$ and $\sigma_i(Y,Z,W)$ as follows:
\[
k\sigma_k(Y,Z,W) \equiv \sum_{j=1}^k(-1)^{j-1}S_j(Y,Z,W) \sigma_{k-1}(Y,Z,W) \,.
\]

\begin{lemma}\label{le:sj}
\[S_{2l+1}(Y,Z,W) \equiv 0,\, 0 \leq l \leq \frac{q-1}{2},\]
\[S_{2l}(Y,Z,W) \equiv -2 \sigma_2^l(Y,Z,W), \, 0 \leq l \leq \frac{q-1}{2}.\]
\end{lemma}
\begin{proof}
Using Corollary~\ref{cor:sigmas}, the Newton identities, the fact that $S_1(Y,Z,W) \equiv \sigma_1(Y,Z,W)$, $\sigma_0 = 1$,  and induction, the lemma follows.
\end{proof}

\begin{lemma}
If $\sigma_2(Y,Z,W)$ is reducible then the set $U$ is extendable.
\end{lemma}
\begin{proof}
Suppose that $\sigma_2(Y,Z,W)$ is reducible. By equation~\eqref{eq:sigma2}, $\sigma_2(y,z,w)$ 
must be a square for any $(y,z,w)$ such that $l(y,z,w)$ meets $\C$ in at least one point. So there 
are triples $(y,z,w)$, contained in a line (the dual of the pencil of lines through a point $P \in \C$) for  
which $\sigma_2(y,z,w)$ is a square. It follows that $\sigma_2(Y,Z,W) = (AY+BZ+CW)^2$. 
Now define $U^* := U \cup \{(A,B,C,1),(-A,-B,-C,1)\}$. Consider any point $P \in \C$ and any line $l
(y,z,w)$ on $P$. From Equation~\eqref{eq:sigma2} it follows that each plane on $l$ now contains 
exactly $q$ points of $U^*$. But if $P$ is a point determined by $U^*$, then there exists a line $m$ 
on  $P$ containing $r \geq 2$ points of $U^*$. But all $q+1$ planes on $m$ contain exactly $q$ 
points of $U^*$, so $q^2 = |U^*|=r+(q+1)(q-r)$, a contradiction. Hence, $U^*$ does not determine 
the points of $\C$.
\end{proof}

\begin{theorem}
If $U$ is not extendable, then $q=p$.
\end{theorem}
\begin{proof}
We define
%\begin{align}
\begin{eqnarray*}
\chi(X,Y,Z,W) & := & \sum_{i=1}^{q^2-2} (X+a_iY+b_iZ+c_iW)^{q-1} \nonumber\\
 & = &  \sum_{i=1}^{q^2-2}\sum_{j=0}^{q-1}{q-1 \choose j}X^{q-1-j} (a_iY+b_iZ+c_iW)^j \nonumber \\
%& = & \sum_{j=0}^{q-1}(-1)^jX^{q-1-j}  \sum_{i=1}^{q^2-2}(a_iY+b_iZ+c_iW)^j  \nonumber\\
& = & \sum_{j=0}^{q-1}(-1)^jX^{q-1-j}  S_j(Y,Z,W)  \nonumber\\
& = & -2 \sum_{k=0}^{\frac{q-1}{2}} X^{q-1-2k} \sigma_2^k(Y,Z,W)  = -2 \frac{X^{q+1} - \sigma_2^{\frac{q+1}{2}}(Y,Z,W)}{X^2-\sigma_2(Y,Z,W)}\,,\label{eq:chi}
%\end{align}
\end{eqnarray*}
where we used Lemma~\ref{le:sj} to obtain the second last equality. If $U$ is not extendable, then $\sigma_2(Y,Z,W)$ is not reducible. So the range of $\sigma_2(Y,Z,W)$ is the complete field $\mathbb{F}_q$, so for each non-square $\nu \in \mathbb{F}_q$, we can find a triple $(y,z,w)$ such that $\sigma_2(y,z,w)=\nu$. Then $\sigma_{2}^{\frac{q+1}{2}}(y,z,w) = -\sigma_2(y,z,w)$ and 
\begin{equation}\label{eq:chix}
 \chi(X,y,z,w) = -2 \frac{X^{q+1}+\sigma_2(y,z,w)}{X^2-\sigma_2(y,z,w)} 
 \end{equation}
It is now easy to see that the range of $\chi(X,Y,Z,W)$ will contain at least $\frac{q+1}{2}$ different elements of $\mathbb{F}_q$. On the other hand, 
\[ \chi(x,y,z,w) = q^2-2 - |U \cap \pi(x,y,z,w)| \mod p\,,\] for any 4-tuple $(x,y,z,w) \not \in \{(1,0,0,0),(0,0,0,0)\}$. So the right hand side is necessarily an element of $\mathbb{F}_p$, a contradiction with the range of $\chi(X,Y,Z,W)$ if $q > p$. 
\end{proof}

\section{The intersection numbers for $q$ a prime}\label{sec:geom}

Suppose now that $q=p$, $p$ an odd prime. We consider the possibilities of $\chi(X,Y,Z,W)$. Consider a plane $\pi(x,y,z,w)$.
% and let $\K$ be the maximal partial ovoid of 
%size $q^2-1$ of $\q(4,q)$ and $U$ the corresponding set of size $q^2-2$ of affine points, 
%representing $\K$ in $T_2(\C)$. We first present an overview of the range of $\chi(X,Y,Z,W)$
\begin{itemize}
\item[(a)] Suppose that $\sigma_2(y,z,w) = 0$. Then $\chi(X,y,z,w) = -2X^{q-1}$, hence $\chi(x,y,z,w)=0$ if $x=0$ and $\chi(x,y,z,w)=-2$ if $x \neq 0$.
\item[(b)] Suppose that $\sigma_2(y,z,w)$ is a square different from $0$. If $x^2 \neq \sigma_2(y,z,w)$ then $\chi(x,y,z,w) = -2$. If $x^2=\sigma_2(y,z,w)$ then $\chi(x,y,z,w) = -1$.
\item[(c)] Suppose that $\sigma_2(y,z,w)$ is a non-square. Then 
\[ \chi(x,y,z,w) = -2 \frac{x^2+\sigma_2(y,z,w)}{x^2-\sigma_2(y,z,w)} \neq 0 \]
\end{itemize}

\begin{lemma}\label{le:dual}
The curve $\sigma_2(Y,Z,W) = 0$ is the dual of $\C$.
\end{lemma}
\begin{proof}
Theorem~\ref{theo2} ensures that the set of lines of $\q(4,q)$, not meeting $\K$, is the set of lines 
of a hyperbolic quadric embedded as a hyperplane section in $\q(4,q)$. We denote this hyperbolic 
quadric as $\q^+$.  Since $\K = \{(\infty)\} \cup U$, clearly $(\infty) \not \in \q^+$, and from the proof 
of Theorem~\ref{theo3} in \cite{PayneThas84}, it follows that $\q^+$ is represented in $T_2(\C)$ as 
a hyperbolic quadric meeting $\pi_{\infty}$ in $\C$. We denote this quadric as $\q_T^+$. The 
hyperbolic 
quadric $\q_T^+$ contains exactly $q+1$ points of type (ii). Consider such a point, represented by 
the plane $\pi$. 
The two lines of type (a) of $\q_T^+$ incident with $\pi$, are contained in $\pi$, and do 
not meet $U$. But the other $q-1$ lines of $T_2(\C)$, incident with $\pi$, do meet $U$ in exactly 
one point. Hence the plane $\pi$ must contain exactly $q-2$ points of $U$. If $\pi$ is 
represented by the 4-tuple $(x,y,z,w)$, then $\chi(x,y,z,w) = q^2-2-|\pi \cap U| \mod q$. So if $|\pi 
\cap U| = q-2$, then $\chi(x,y,z,w)=0$ and by the above overview of the range of $\chi$, 
the planes $\pi(x,y,z,w)$ that represent a point of type (ii) of $\q_T^+$, are exactly those 
for which $\sigma_2(y,z,w) = 0 = x$. But the planes that represent points of type (ii) of 
$\q_T^+$ are planes that meet $\C$ in a tangent line. Hence, 
$\sigma_2(y,z,w)=0$ if and only if $l(y,z,w)$ is a tangent line to $\C$. 
\end{proof}

\begin{corollary}
A plane $\pi(x,y,z,w)$ represents an elliptic quadric containing $(\infty)$ if and only if $\sigma_2
(y,z,w)$ is a non-square.
\end{corollary}
\begin{proof}
From the proof of Theorem~\ref{theo3}, it follows that an elliptic quadric containing $(\infty)$ is 
represented in $T_2(\C)$ by a plane meeting $\pi_{\infty}$ in a line external to $\C$. The Corollary 
now follows from Lemma~\ref{le:dual}.
\end{proof}

\begin{corollary}\label{cor:el}
If an elliptic quadric $\q^- \subset \q(4,q)$ contains a point of $\K$, then
\[
|\q^- \cap \K|  \mod p \in \{-1+2\frac{x^2+\nu}{x^2-\nu}\,\|\, \mbox{\rm $\nu$ running over the 
non-squares, $x \in \mathbb{F}_q$}\}
\]
\end{corollary}

\begin{proof}
If an elliptic quadric contains a point of $\K$, we can choose it to be the point $(\infty)$. Then 
\begin{equation}\label{eq:gen}
|\pi(x,y,z,w) \cap U| \mod q = -2 - \chi(x,y,z,w) = -2 + 2\frac{x^2+\nu}{x^2-\nu}\,,
\end{equation}
$\nu = \sigma_2(y,z,w)$, which is non-square.
\end{proof}

Consider now any point $P \in \q(4,q) \setminus \q^+$. Then $P^\perp \cap \q^+$ is a conic $C_P$, 
and $C_P^\perp = \{P,P'\}$, $P \neq P' \in \q(4,q) \setminus \q^+$. We call $P'$ the antipode of $P$.
Consider now the point $(\infty)$, this is collinear with the points of type (ii) of $\q_T^+$. But for 
each  point of type (ii) of $\q_T^+$, represented by a plane $\pi(x,y,z,w)$, we have seen that $x=0$. 
Hence the point $(0,0,0,1)$ is contained in the planes representing the points of type (ii) of $\q_T^
+$, so, the points of type (ii) of  $\q_T^+$ are collinear with $(0,0,0,1)$. Hence, the point 
$(0,0,0,1)$ 
is the antipode of the point $(\infty)$.

\begin{lemma}
If an elliptic quadric $\q^- \subset \q(4,q)$ contains a point of $\K$ and its antipode, then
$|\q^- \cap \K|  \equiv -3 \mod q$.
\end{lemma}
\begin{proof}
A point and its antipode are non-collinear, and the collineation group of $\q(4,q)$ acts transitively 
on the pairs of non-collinear points. So in the $T_2(\C)$ representation, if an elliptic quadric 
contains a point of $\K$, this can be chosen $(\infty)$ while its antipode can be chosen to be the 
point $(0,0,0,1)$. For a plane $\pi(x,y,z,w)$ containing $(0,0,0,1)$, we have $x=0$. The lemma 
now follows from Corollary~\ref{cor:el}.
\end{proof}

We remark that the computed intersection numbers (modulo $q$) do not exclude elliptic quadrics 
that contain no point of $\K$. We list the range for intersection numbers modulo $q$ found in 
Corollary~\ref{cor:el} for $q \in \{5,7,11\}$. Recall that these numbers are valid for elliptic quadrics 
containing at least one point of $\K$. Hence $0$ means a positive multiple of $q$ in reality.
\begin{itemize}
\item $q=5$: $\{0,2,3\}$
\item $q=7$: $\{2,3,4,6\}$
\item $q=11$: $\{0,4,5,8,9,10\}$.
\end{itemize}
We used an explicit description of the known examples (\cite{kc}) to compute the intersection 
numbers with all elliptic quadrics. We list the results. In this list, for $q=5$ and $q=11$, we see that 
there are elliptic quadrics containing no point of $\K$. However this $0$ is {\bf not} related to a $0$ 
in the above list.
\begin{itemize}
\item $q=5$: $\{0,2,3,5,8,12\}$
\item $q=7$: $\{2,3,4,6,9,10,18\}$
\item $q=11$: $\{0,4,5,8,9,10,11,15,16,20,30\}$.
\end{itemize}
As a final remark, we notice that the number of different intersection numbers is relatively large 
compared with $q$. On the other hand, an elliptic quadric containing a point of $\K$ and its 
antipode always meets $\K$ in $-3 \mod q$ points. In the above list, we notice for each $q$ only 
two different intersection numbers corresponding to $-3 \mod q$.  This might suggest that pairs 
point-antipode play a special role, and indeed, for the known examples, it is true that when a point 
belongs to $\K$, then also its antipode belongs to $\K$, \cite[Theorem 12]{kc}. Unfortunately, the above combinatorial information seems too weak to prove such a characterisation. It is our feeling that such a characterisation could be helpful in proving the non-existence for larger $q$. We note
that in \cite{kc}, where a completely different approach is used, a comparable conclusion 
on the pairs point-antipode is made. Finally, we also mention the work in \cite{DWT}, where 
the non-existence for larger $q$ is shown under the extra assumption that $(q^2-1)^2$ 
divides the order of the automorphism group of the maximal partial ovoid. 

\section*{Acknowledgement}

The author thanks the department of Computer Science at E\"otv\"os Lor\'and University in 
Budapest, and especially P\'eter Sziklai, Tam\'as Sz\H{o}nyi and Zsuzsa Weiner for their hospitality.

%\bibliography{art22}

\end{document}